\documentclass[psamsfonts]{amsart}  
\usepackage{amssymb}

\usepackage[dvips]{graphicx}

\usepackage{graphicx}
\usepackage{amssymb}                       
\usepackage{amsmath}
\usepackage{color}
\usepackage{times}

\usepackage[unicode,bookmarks,colorlinks]{hyperref}
\hypersetup{
   linkcolor=brickred,
}

\definecolor{mahogany}{cmyk}{0, 0.77, 0.87, 0}
\definecolor{salmon}{cmyk}{0, 0.53, 0.38, 0}
\definecolor{melon}{cmyk}{0, 0.46, 0.50, 0}
\definecolor{yellowgreen}{cmyk}{0.44, 0, 0.74, 0}
\definecolor{brickred}{cmyk}{0, 0.89, 0.94, 0.28}
\definecolor{OliveGreen}{cmyk}{0.64, 0, 0.95, 0.40}
\definecolor{RawSienna}{cmyk}{0, 0.72, 1.0, 0.45}
\definecolor{ZurichRed}{rgb}{1, 0, 0} 

\usepackage{fancyhdr}
\usepackage{amsmath,amstext,amssymb,amsopn,amsthm}
\usepackage{amsmath,amssymb,amsthm}
\usepackage[mathscr]{eucal}

\begin{document}

\newtheorem{lemma}[thm]{Lemma}
\newtheorem{proposition}{Proposition}
\newtheorem{theorem}{Theorem}[section]
\newtheorem{deff}[thm]{Definition}
\newtheorem{case}[thm]{Case}
\newtheorem{prop}[thm]{Proposition}
\newtheorem{example}{Example}

\newtheorem{corollary}{Corollary}

\theoremstyle{definition}
\newtheorem{remark}{Remark}

\numberwithin{equation}{section}
\numberwithin{definition}{section}
\numberwithin{corollary}{section}

\numberwithin{theorem}{section}

\numberwithin{remark}{section}
\numberwithin{example}{section}
\numberwithin{proposition}{section}

\newcommand{\gap}{\lambda_{2,D}^V-\lambda_{1,D}^V}
\newcommand{\gapR}{\lambda_{2,R}-\lambda_{1,R}}
\newcommand{\bD}{\mathrm{I\! D\!}}
\newcommand{\calD}{\mathcal{D}}
\newcommand{\calA}{\mathcal{A}}

\newcommand{\conjugate}[1]{\overline{#1}}
\newcommand{\abs}[1]{\left| #1 \right|}
\newcommand{\cl}[1]{\overline{#1}}
\newcommand{\expr}[1]{\left( #1 \right)}
\newcommand{\set}[1]{\left\{ #1 \right\}}

\newcommand{\calC}{\mathcal{C}}
\newcommand{\calE}{\mathcal{E}}
\newcommand{\calF}{\mathcal{F}}
\newcommand{\Rd}{\mathbb{R}^d}
\newcommand{\BR}{\mathcal{B}(\Rd)}
\newcommand{\R}{\mathbb{R}}

\newcommand{\al}{\alpha}
\newcommand{\RR}[1]{\mathbb{#1}}
\newcommand{\bR}{\mathrm{I\! R\!}}
\newcommand{\ga}{\gamma}
\newcommand{\om}{\omega}
\newcommand{\A}{\mathbb{A}}
\newcommand{\bH}{\mathbb{H}}

\newcommand{\bb}[1]{\mathbb{#1}}
\newcommand{\bI}{\bb{I}}
\newcommand{\bN}{\bb{N}}

\newcommand{\uS}{\mathbb{S}}
\newcommand{\M}{{\mathcal{M}}}
\newcommand{\calB}{{\mathcal{B}}}

\newcommand{\W}{{\mathcal{W}}}

\newcommand{\m}{{\mathcal{m}}}

\newcommand {\mac}[1] { \mathbb{#1} }

\newcommand{\bC}{\Bbb C}

\newtheorem{rem}[theorem]{Remark}
\newtheorem{dfn}[theorem]{Definition}
\theoremstyle{definition}
\newtheorem{ex}[theorem]{Example}
\numberwithin{equation}{section}

\newcommand{\Pro}{\mathbb{P}}
\newcommand\F{\mathcal{F}}
\newcommand\E{\mathbb{E}}
\newcommand\e{\varepsilon}
\def\H{\mathcal{H}}
\def\t{\tau}

\title[Weighted inequalities]{Weighted norm inequalities  for fractional maximal operators--a Bellman function approach}

\author{Rodrigo Ba\~nuelos}\thanks{R. Ba\~nuelos is supported in part  by NSF Grant
\# 0603701-DMS}
\address{Department of Mathematics, Purdue University, West Lafayette, IN 47907, USA}
\email{banuelos@math.purdue.edu}
\author{Adam Os\c ekowski}\thanks{A. Os\c ekowski is supported in part by the NCN grant DEC-2012/05/B/ST1/00412.}
\address{Department of Mathematics, Informatics and Mechanics, University of Warsaw, Banacha 2, 02-097 Warsaw, Poland}
\email{ados@mimuw.edu.pl}


\begin{abstract}
We study classical weighted $L^p\to L^q$ inequalities  for the fractional maximal operators on $\R^d$, proved originally by Muckenhoupt and Wheeden in the 70's. 
We establish a slightly stronger version of this inequality with the use of a novel extension of Bellman function method. More precisely, the estimate is deduced from the existence of a certain special function which enjoys appropriate majorization and concavity. From this result and an  explicit version of the ``$A_{p-\varepsilon}$ theorem," derived  also with Bellman functions, we obtain the sharp inequality of Lacey, Moen, P\'erez and Torres. 
\end{abstract}

\maketitle

\section{Introduction}
The motivation for the results of this paper comes from the question about weighted $L^p\to L^q$-norm inequalities  for fractional maximal operators 
 on $\R^d$, with the constants of optimal order. To introduce the necessary background and notation, let $0\leq \alpha<d$ be a fixed constant. The fractional maximal operator $\mathcal{M}^\alpha$ is given by
$$ \mathcal{M}^\alpha\varphi(x)=\sup\left\{|Q|^{\frac{\alpha}{d}-1}\int_Q |\varphi(u)|\mbox{d}u:Q\subset \R^d\mbox{ is a cube containing }x\right\},$$
where $\varphi$ is a locally integrable function on $\R^d$, $|Q|$ denotes the Lebesgue measure of $Q$ and the cubes we consider above have  sides parallel to the axes. In particular, if we put $\alpha=0$, then $\mathcal{M}^\alpha$ reduces to the classical Hardy-Littlewood maximal operator.  
The above fractional operators play an important role in analysis and PDEs, and form a convenient tool to study differentiability properties of functions. In particular, it is often of interest to obtain optimal, or at least good bounds for norms of these operators. 
We will be mostly interested in the weighted setting. In what follows, the  word ``weight'' refers to a locally integrable, positive function  on $\R^d$, which will usually be denoted by $w$. Given $p\in (1,\infty)$, we say that $w$ belongs to the Muckenhoupt $A_p$ class (or, in short, that $w$ is an $A_p$ weight), if the $A_p$ characteristics $[w]_{A_p}$, given by
$$ [w]_{A_p}:=\sup_Q \left(\frac{1}{|Q|}\int_Q w \right)\left(\frac{1}{|Q|}\int_Q w^{-1/(p-1)}\right)^{p-1},$$
is finite. 
One can also define the appropriate versions of this condition for $p=1$ and $p=\infty$, by passing above with $p$ to the appropriate limit (see e.g. \cite{Gr}, \cite{Hr}). However, we omit the details, as in this paper we will be mainly concerned with the case $1<p<\infty$.

As shown  by Muckenhoupt \cite{M}, the $A_p$ condition arises naturally during the study of weighted $L^p$ bounds for the Hardy-Littlewood maximal operator. To be more precise, for a given $1<p<\infty$, the inequality 
$$ ||\mathcal{M}^0\varphi||_{L^p(w)}\leq C_{p,d,w}||\varphi||_{L^p(w)}$$
holds true with some constant $C_{p,d,w}$ depending only on the parameters indicated, if and only if $w$ is an $A_p$ weight. Here, of course,
$ ||\varphi||_{L^p(w)}=\left(\int_{\R^d} |\varphi|^pw\mbox{d}u\right)^{1/p}$ is the usual norm in the weighted $L^p$ space. This result was extended to the fractional setting by Muckenhoupt and Wheeden \cite{MW}. Let $p,\,q$ be positive exponents satisfying the relation $\frac{1}{q}=\frac{1}{p}-\frac{\alpha}{d}$. Then the inequality
$$ \left(\int_{\R^d} \big(\mathcal{M}^\alpha \varphi(x)\big)^qw(x)^q \mbox{d}x\right)^{1/q}\leq C_{p,\alpha,d,w}\left(\int_{\R^d} |\varphi(x)|^pw(x)^p\mbox{d}x\right)^{1/p}$$
if and only if 
$$ \sup_Q \left(\frac{1}{|Q|}\int_Q w^q\right)\left(\frac{1}{|Q|}\int_Q w^{-p'}\right)^{q/p'}<\infty,$$
where $p'=p/(p-1)$ is the harmonic conjugate to $p$. In other words, passing to $w^q$, we see that 
$$ ||\mathcal{M}^\alpha||_{L^p(w^{p/q})\to L^q(w)}\leq C_{p,\alpha,d,w}$$
 if and only if $w\in A_{q/p'+1}.$ 
 
Actually, one can choose the above constants $C_{p,d,w}$ and $C_{p,\alpha,d,w}$ so that the dependence on $w$ is through the appropriate characteristics of the weight only. Then there arises a very interesting question, concerning the sharp description of this dependence. The first result in this direction, going back to early 90's, is that of Buckley \cite{Bu}. Specifically, he proved that Hardy-Littlewood operator satisfies
\begin{equation}\label{Buckley}
 ||\mathcal{M}^0\varphi||_{L^p(w)}\leq C_{p,d}[w]_{A_p}^{1/(p-1)}||\varphi||_{L^p(w)}
\end{equation}
and showed that the power $1/(p-1)$ cannot be decreased in general. 
By now, there are several different proofs of this inequality (which produce various upper bounds for the involved constant $C$). For instance, we refer the interested reader to works of Coifman and Fefferman \cite{CF}, Lerner \cite{Le}, Nazarov and Treil \cite{NT} and, for a slightly stronger statement, to the recent paper of Hyt\"onen and Perez \cite{HP}. In the fractional setting, Lacey et al. \cite{L} proved that
\begin{equation}\label{Lacey}
 ||\mathcal{M}^\alpha\varphi||_{L^q(w)}\leq C_{p,\alpha, d}[w]_{A_{q/p'+1}}^{(1-\alpha/d)p'/q}||\varphi||_{L^p(w^{p/q})},
\end{equation}
and the exponent $(1-\alpha/d)p'/q$ cannot be improved. 
See also the recent paper of Cruz-Uribe and Moen \cite{CM} for certain generalizations of the above result.

The purpose of this paper is to provide yet another extensions of \eqref{Buckley} and \eqref{Lacey}. Here is the precise statement.

\begin{theorem}\label{mainthm}
Suppose that $0\leq \alpha<d$ is fixed and let $p$, $q$ be exponents satisfying $\frac{1}{q}=\frac{1}{p}-\frac{\alpha}{d}$. 
Then for any $A_{q/p'+1}$ weight $w$ and any locally integrable function $\varphi$ on $\R^d$ we have
\begin{equation}\label{mainin0}
 ||\mathcal{M}^\alpha\varphi||_{L^q(w)}\leq C^d\inf_{1< r<q/p'+1}\left\{[w]_{A_r}^{1/(q-s)}\left(\frac{q}{s}\right)^{1/(q-s)}\right\}||\varphi||_{L^p(w^{p/q})},
\end{equation}
where $C$ is an absolute constant (for instance, $C=12$ works fine) and
\begin{equation}\label{defs}
 s=s(r,p,\alpha)=q-\frac{rq^2}{(r-1)(q-p)+pq}.
\end{equation}
\end{theorem}

To see that this result implies \eqref{Buckley} and \eqref{Lacey}
, we will need to establish the following version of the ``$A_{p-\varepsilon}$ theorem" which is sharper than what the classical results give, see Coifman and Fefferman \cite{CF}. 

\begin{theorem}\label{weights}
Suppose that $w$ is an $A_\beta$ weight ($1<\beta<\infty$) and put
\begin{equation}\label{defr}
 r=\frac{\beta(1+2^{\beta d/(\beta-1)}\beta^{\beta/(\beta-1)}[w]_{A_\beta}^{1/(\beta-1)})}{\beta+2^{\beta d/(\beta-1)}\beta^{\beta/(\beta-1)}[w]_{A_\beta}^{1/(\beta-1)}} \in (1,\beta).
\end{equation}
Then $w$ is an $A_r$ weight and $[w]_{A_r}\leq 2^{rd+r}[w]_{A_\beta}^{(r-1)/(\beta-1)}.$
\end{theorem}

Now \eqref{Buckley} and \eqref{Lacey} 
 follow easily: apply Theorem \ref{weights} with $\beta=q/p'+1$ and plug $r$ given by \eqref{defr} into \eqref{mainin0}. We have
$$ \frac{q}{s}=\frac{(r-1)\left(\frac{q}{p}-1\right)+q}{\beta-r}$$
and
$$ \beta-r=\frac{\beta(\beta-1)}{\beta+2^{\beta d/(\beta-1)}\beta^{\beta/(\beta-1)}[w]_{A_\beta}^{1/(\beta-1)}},$$
so $q/s\leq c_{p,\alpha,d}[w]_{A_\beta}^{1/(\beta-1)}$. Therefore,
\begin{align*}
 [w]_{A_r}^{1/(q-s)}\left(\frac{q}{s}\right)^{1/(q-s)}&\leq 2^{(rd+r)/(q-s)}[w]_{A_\beta}^{(r-1)/(\beta-1)(q-s)}\cdot c_{p,\alpha,d}^{1/(q-s)}[w]_{A_\beta}^{1/(\beta-1)(q-s)}\\
 & \leq C_{p,\alpha,d}[w]_{A_\beta}^{r/(\beta-1)(q-s)},
\end{align*}
and it remains to note that
$$ \frac{r}{(\beta-1)(q-s)}=\frac{(r-1)(q-p)+pq}{(\beta-1)q^2}\leq \frac{(\beta-1)(q-p)+pq}{(\beta-1)q^2}=\left(1-\frac{\alpha}{d}\right)\frac{p'}{q}.$$
A few words about our approach are in order. Actually, we will work with the dyadic version of the fractional maximal operator, given by 
$$ \mathcal{M}_d^\alpha\varphi(x)=\sup\left\{|Q|^{\frac{\alpha}{d}-1}\int_Q |\varphi(u)|\mbox{d}u:Q\subset \R^d\mbox{ is a dyadic cube containing } x\right\}.$$
Here the dyadic cubes are those formed by the grids $2^{-N}\mathbb{Z}^d$, $N\in\mathbb{Z}$. By a standard comparison of the sizes of $\mathcal{M}^\alpha\varphi$ and $\mathcal{M}^\alpha_d\varphi$ (see e.g. Grafakos \cite{Gr} or Stein \cite{St}), we will be done if we establish the following version of Theorem \ref{mainthm}.

\begin{theorem}\label{mainth}
Suppose that $0\leq \alpha<d$ is fixed and let $p$, $q$ be exponents satisfying $\frac{1}{q}=\frac{1}{p}-\frac{\alpha}{d}$. 
Then for any $A_{q/p'+1}$ weight $w$ and any locally integrable function $\varphi$ on $\R^d$ we have
\begin{equation}\label{mainin}
 ||\mathcal{M}^\alpha_d\varphi||_{L^q(w)}\leq \inf_{1< r<q/p'+1}\left\{[w]_{A_r}^{1/(q-s)}\left(\frac{q}{s}\right)^{1/(q-s)}\right\}||\varphi||_{L^p(w^{p/q})}.
\end{equation}
\end{theorem}

\newcommand\ddedd{Theorem  \ref{mainth} holds in the general setting of martingales under mild regularity assumptions on the filtration.   Indeed, martingale weighted norm inequalities have been extensively studied by many authors since the 70's. In particular, a  version of Theorem \ref{mainth} for martingales was proved  by Izumisawa and Kazamaki in  \cite{IK}.  Their proof, described at the end of this paper, is extremely simple (nothing more than Doob's inequality) and already gives the same bound as Theorem \ref{mainth}.  Since Theorem \ref{mainth} can be obtained from the martingale inequality for dyadic martingales, so then can Theorem \ref{mainthm}. However,  an interesting feature of our approach here is that we will \emph{not} exploit Calder\'on-Zygmund-type decomposition/stopping time arguments, covering lemmas, Carlesson embedding theorem, Doob's inequality, etc., which are the typical tools in the proofs of such results. Instead, the statement will be established directly with the use of the so-called Bellman function method which, in the setting of these type of harmonic analysis problems, has its roots in the work of Burkholder \cite{Bur1} on sharp martingale inequalities. More precisely, Theorem \ref{mainth} and Theorem \ref{weights}  will be deduced from the existence of special functions, possessing certain majorization and concavity properties.  These type of techniques have been widely used in recent years to obtain optimal bounds for various operators in analysis (see \cite{Hyt, SlaVas} for some of this literature).  For the probabilistic point of view and its many applications to martingale inequalities  we refer to the second author's monograph \cite{Ose1} and the many references to applications provided therein.  }

The paper is organized as follows. In the next section we establish Theorem \ref{weights}. Section~3 is devoted to the proof of Theorem \ref{mainth}. Both these statements are shown with the use of the so-called Bellman function method, a powerful tool exploited widely in analysis and probability, and its certain enhancement. We end with some remarks concerning martingale versions of the above weighted norm inequalities.

\section{Proof of Theorem \ref{weights}}
Throughout this section, $\beta\in (1,\infty)$, $d\geq 1$ and $c\geq 1$ are fixed parameters. 
Let us define the hyperbolic domain
$$ \Omega_c=\{(w,v)\in \bR_{+} \times \bR_{+}: 1\leq wv^{\beta-1}\leq c\}.$$
We start with the following geometric fact.
\begin{lemma}\label{geomlemma}
Let $\alpha\in [2^{-d},1-2^{-d}]$ and suppose that points $\mathtt{P}$, $\mathtt{Q}$ and $\mathtt{R}=\alpha \mathtt{P}+(1-\alpha)\mathtt{Q}$ lie in $\Omega_c$. Then the whole line segment $\mathtt{P}\mathtt{Q}$ is contained within $\Omega_{2^{\beta d}c}$.
\end{lemma}
\begin{proof}
Using a simple geometrical argument, it is enough to consider the case when the points  $\mathtt{P}$ and $\mathtt{R}$ lie on the curve $wv^{\beta-1}=c$ (the upper boundary of $\Omega_c$) and $\mathtt{Q}$ lies on the curve $wv^{\beta-1}=1$ (the lower boundary of $\Omega_c$). Then the line segment $\mathtt{R}\mathtt{Q}$ is contained within $\Omega_c$, and hence also within $\Omega_{2^{\beta d}c}$, so  it is enough to ensure that the segment $\mathtt{P}\mathtt{R}$ is contained in $\Omega_{2^{\beta d}c}$. Let $\mathtt{P}=(\mathtt{P}_x,\mathtt{P}_y)$, $\mathtt{Q}=(\mathtt{Q}_x,\mathtt{Q}_y)$ and $\mathtt{R}=(\mathtt{R}_x,\mathtt{R}_y)$. We consider two cases. If $\mathtt{P}_x<\mathtt{R}_x$, then 
$$ \mathtt{P}_y=\alpha^{-1}\mathtt{R}_y-(\alpha^{-1}-1)\mathtt{Q}_y<2^d\mathtt{R}_y,$$
so the segment $\mathtt{P}\mathtt{R}$ is contained in the quadrant $\{(x,y):x\leq \mathtt{R}_x,\,y\leq 2^d\mathtt{R}_y\}$. Consequently, $\mathtt{P}\mathtt{R}$ lies below the hyperbola $xy^{\beta-1}=\gamma$ passing through $(\mathtt{R}_x,2^d\mathtt{R}_y)$,  that is,  corresponding to $\gamma=2^{d(\beta-1)}R_xR_y^{\beta-1}\leq 2^{\beta d}c$. This proves the assertion in the case $\mathtt{P}_x<\mathtt{R}_x$. In the case $\mathtt{P}_x\geq\mathtt{R}_x$ the reasoning is similar. Indeed, we check easily that the line segment $\mathtt{P}\mathtt{R}$ lies below the hyperbola $xy^{\beta-1}=\gamma'$ passing through $(2^d\mathtt{R}_x,\mathtt{R}_y)$.  That is,  the one with $\gamma'=2^dR_xR_y^{\beta-1}\leq 2^{\beta d}c$.
\end{proof}

The proof of Theorem \ref{weights} will rest on a Bellman function which was invented by Vasyunin in \cite{V} during the study of power estimates for $A_\beta$ weights on the real line. To recall this object, we need some more notation. Let $s$ be the unique negative number satisfying the equation
$$ (1-s)(1-s/\beta)^{-\beta}=2^{-\beta d}/c$$
(the existence and uniqueness of $s$ is clear: the function $F(u)=(1-u)(1-u/\beta)^{-\beta}$ is strictly increasing on $(-\infty,0]$ and satisfies $\lim_{u\to -\infty}F(u)=0$, $F(0)=1$). For $r\in \big(\beta(1-s)/(\beta-s),\beta\big)$, define
$$ C=C_{\beta,d,r,c}=(2^{\beta d}c)^{r/(\beta(r-1))}(1-s)^{(\beta-r)/(\beta(1-r))}\left[1+\frac{\beta-r}{\beta(r-1)}s\right]^{-1}.$$
This is precisely the constant $C_{\max}(\beta,-(r-1)^{-1},2^{\beta d}c)$ appearing on p. 50 in \cite{V}. A function defined on a subset of $\R^2$ is said to be locally concave if it is concave along any line segment contained in its domain.

We will need the following lemma from \cite{V}.

\begin{lemma}\label{Vasyunin}
There is a locally concave function $B$ on $\Omega_{2^{pd}c}$ which satisfies 
\begin{equation}\label{bound0}
B(w,v)=w^{-1/(r-1)}\qquad \mbox{ if }wv^{\beta-1}=1
\end{equation}
and such that
\begin{equation}\label{bound}
 B(w,v)\leq C_{\beta,d,r,c}w^{-1/(r-1)} \qquad \mbox{for all }(w,v)\in \Omega_{2^{\beta d}c}.
\end{equation}
\end{lemma}

Combining this with Lemma \ref{geomlemma} and a straightforward induction argument, we see that the above function $B$ has the following property. If $\mathtt{P}_1$, $\mathtt{P}_2$, $\ldots$, $\mathtt{P}_{2^d}$ and $\mathtt{P}=2^{-d}(\mathtt{P}_1+\mathtt{P}_2+\ldots+\mathtt{P}_{2^d})$ lie in $\Omega_c$, then
\begin{equation}\label{convB}
 B(\mathtt{P})\geq 2^{-d}\sum_{k=1}^{2^d}B(\mathtt{P}_k).
\end{equation}

We turn to the proof of Theorem \ref{weights}. Let us first establish two auxiliary facts.

\begin{lemma}
The number $s$ satisfies the double inequality
\begin{equation}\label{double}
-2^{\beta d/(\beta-1)}\beta^{\beta/(\beta-1)}c^{1/(\beta-1)}\leq s\leq -\beta^{\beta/(\beta-1)}c^{1/(\beta-1)}.
\end{equation}
\end{lemma}
\begin{proof}
Let us first focus on the right-hand side inequality. By the definition of $s$, we see that we must prove that
$$ (1+\beta^{\beta/(\beta-1)}c^{1/(\beta-1)})(1+\beta^{1/(\beta-1)}c^{1/(\beta-1)})^{-\beta}\geq 2^{-\beta d}/c,$$
or
$$ (1+\beta^{1/(\beta-1)}c^{1/(\beta-1)})^\beta\leq 2^{\beta d}c(1+\beta^{\beta/(\beta-1)}c^{1/(\beta-1)}).$$
But this is simple and follows immediately from the estimate
$$ (1+\beta^{1/(\beta-1)}c^{1/(\beta-1)})^\beta\leq (2\beta^{1/(\beta-1)}c^{1/(\beta-1)})^\beta\leq 2^{\beta d}\beta^{\beta/(\beta-1)}c^{\beta/(\beta-1)}.$$
We turn our attention to the left-hand side inequality in \eqref{double}. This time we must show that
$$ (1+2^{\beta d/(\beta-1)}\beta^{1/(\beta-1)}c^{1/(\beta-1)})^\beta\geq 2^{\beta d}c(1+2^{\beta d/(\beta-1)}\beta^{\beta/(\beta-1)}c^{1/(\beta-1)}).$$
Using the bound $(1+a)^\beta\geq a^\beta+a^{\beta-1}$, we may write
\begin{align*}
 (1+2^{\beta d/(\beta-1)}\beta^{1/(\beta-1)}c^{1/(\beta-1)})^\beta&\geq 2^{\beta^2d/(\beta-1)}\beta^{\beta/(\beta-1)}c^{\beta/(\beta-1)}+2^{\beta d}\beta c\\
 &=2^{\beta d}c(\beta+2^{\beta d/(\beta-1)}\beta^{\beta/(\beta-1)}c^{1/(\beta-1)}),
\end{align*}
and the proof is finished.
\end{proof}

\begin{lemma}\label{aux2}
Suppose that $w$ is an $A_\beta$ weight with $[w]_{A_\beta}=c$. Then for any $r\in \big(\beta(1-s)/(\beta-s),\beta\big)$ we have $ [w]_{A_r}\leq C_{\beta,d,r,c}^{r-1}.$
\end{lemma}
\begin{proof}
Fix a cube $\mathfrak{Q}\subset \R^d$ and consider the family of its ``dyadic'' subcubes. That is, for a given $n\geq 0$, let $\mathcal{Q}^n$ be the collection of  $2^{nd}$ pairwise disjoint cubes contained in $\mathfrak{Q}$, each of which has measure $2^{-nd}|\mathfrak{Q}|$. 
 Let $(\mathtt{w}_n)_{n\geq 0}$ and $(\mathtt{v}_n)_{n\geq 0}$ be the conditional expectations of $w$ and $w^{-1/(\beta-1)}$ with respect to $(\mathcal{Q}^n)_{n\geq 0}$. That is, for any $x\in \mathfrak{Q}$ and any nonnegative integer $n$, define
$$
\mathtt{w}_n(x)=\frac{1}{|Q|}\int_Q w\quad \mbox{ and }\quad \mathtt{v}_n(x)=\frac{1}{|Q|}\int_Q w^{-1/(\beta-1)},
$$
where $Q$ is the unique element of $\mathcal{Q}^n$ which contains $x$. Directly from this definition, we see that $\mathtt{w}_n$, $\mathtt{v}_n$ are constant on each $Q\in\mathcal{Q}^n$ and we have
$$ \mathtt{w}_n|_Q=\frac{1}{2^d}\sum_{R\subset Q,\,R\in \mathcal{Q}^{n+1}} \frac{1}{|R|}\int_R w=\frac{1}{2^d}\sum_{R\subset Q,\,R\in \mathcal{Q}^{n+1}}\mathtt{w}_{n+1}|_R$$
and
$$ \mathtt{v}_n|_Q=\frac{1}{2^d}\sum_{R\subset Q,\,R\in \mathcal{Q}^{n+1}} \frac{1}{|R|}\int_R w^{-1/(p-1)}=\frac{1}{2^d}\sum_{R\subset Q,\,R\in \mathcal{Q}^{n+1}}\mathtt{v}_{n+1}|_R.$$
Therefore, by \eqref{convB}, we see that
$$ \int_Q B(\mathtt{w}_{n+1},\mathtt{v}_{n+1})\mbox{d}u\leq \int_Q B(\mathtt{w}_{n},\mathtt{v}_{n})\mbox{d}u$$
for all $n=0,\,1,\,2,\,\ldots$, and hence, summing over all $Q\in \mathcal{Q}^n$, we get
$$ \int_\mathfrak{Q} B(\mathtt{w}_{n+1},\mathtt{v}_{n+1})\mbox{d}u\leq \int_\mathfrak{Q} B(\mathtt{w}_{n},\mathtt{v}_{n})\mbox{d}u.$$
Consequently, by induction, we obtain
\begin{align*}
 \int_\mathfrak{Q} B(\mathtt{w}_{n},\mathtt{v}_n)\mbox{d}u\leq \int_\mathfrak{Q} B(\mathtt{w}_0,\mathtt{v}_0)\mbox{d}u&=|\mathfrak{Q}|B\left(\frac{1}{|\mathfrak{Q}|}\int_\mathfrak{Q} w\mbox{d}u,\frac{1}{|\mathfrak{Q}|}\int_\mathfrak{Q} w^{-1/(\beta-1)}\mbox{d}u\right)\\
 &\leq  C_{\beta,d,r,c}|\mathfrak{Q}| \left(\frac{1}{|\mathfrak{Q}|}\int_\mathfrak{Q} w\mbox{d}u\right)^{-1/(r-1)},
 \end{align*}
 where in the last passage we have exploited \eqref{bound}. Now if we let $n\to \infty$, then $\mathtt{w}_n\to w$ and $\mathtt{v}_n\to w^{-1/(\beta-1)}$ almost everywhere, in view of Lebesgue's differentiation theorem. Therefore the above estimate, combined with Fatou's lemma and \eqref{bound0}, gives
$$ \int_\mathfrak{Q} w^{-1/(r-1)}\mbox{d}u=\int_\mathfrak{Q} B(w,w^{-1/(\beta-1)})\mbox{d}u\leq C_{\beta,d,r,c}|\mathfrak{Q}|\left(\frac{1}{|\mathfrak{Q}|}\int_\mathfrak{Q} w\mbox{d}u\right)^{-1/(r-1)}.$$
Multiplying both sides by $ |\mathfrak{Q}|^{-1}\left(\frac{1}{|\mathfrak{Q}|}\int_\mathfrak{Q} w\mbox{d}u\right)^{1/(r-1)}$ and taking the supremum over all $\mathfrak{Q}$, we obtain the desired upper bound for $[w]_{A_r}$.
\end{proof}

\begin{proof}[Proof of Theorem \ref{weights}] Put $c=[w]_{A_\beta}$. 
Let us first note that the number $r$ defined in \eqref{defr} belongs to the interval $\big(\beta(1-s)/(\beta-s),\beta\big)$. Indeed, the inequality $r<\beta$ is evident, so all we need is the lower bound $r>\beta(1-s)/(\beta-s)$. After some easy manipulations, it can be transformed into
$$ s>-2^{\beta d/(\beta-1)+1}\beta^{\beta/(\beta-1)}c^{1/(\beta-1)},$$
which follows from the left inequality in \eqref{double}. Thus we are allowed to apply Lemma \ref{aux2} with the above choice of $r$ and therefore we will be done if we give the appropriate upper bound for 
$$ C_{\beta,d,r,c}^{r-1}=(2^{\beta d}c)^{r/\beta}(1-s)^{(r-\beta)/\beta}\left[1+\frac{\beta-r}{\beta(r-1)}s\right]^{1-r}.$$
 Let us analyze the three factors appearing in this expression. We have
$$ (2^{\beta d}c)^{r/\beta}=2^{rd}c^{r/\beta},$$
and, since $(r-\beta)/\beta<0$, the right inequality in \eqref{double} gives
$$ (1-s)^{(r-\beta)/\beta}\leq (-s)^{(r-\beta)/\beta}\leq (\beta^{\beta/(\beta-1)}c^{1/(\beta-1)})^{(r-\beta)/\beta}\leq c^{(r-\beta)/(\beta(\beta-1))}.$$
Finally, by the left inequality in \eqref{double}, we get
$$ 1+\frac{\beta-r}{\beta(r-1)}s\geq 1-\frac{\beta-r}{\beta(r-1)}\cdot 2^{\beta d/(\beta-1)}\beta^{\beta/(\beta-1)}c^{1/(\beta-1)}=\frac{1}{2},$$
and hence
$$ \left[1+\frac{\beta-r}{\beta(r-1)}s\right]^{1-r}<2^r.$$
Putting all the above facts together, we obtain $[w]_{A_r}\leq 2^{rd+r}c^{(r-1)/(\beta-1)}$, which is the desired claim.
\end{proof}

\section{Weighted inequalities for fractional maximal operator}
Let $0<\alpha<d$ be fixed, and let $p,\,q$ be given two parameters satisfying $\frac{1}{q}=\frac{1}{p}-\frac{\alpha}{d}$. Next, put $u=\frac{p-1}{p}q+1$ and let $r$ be an arbitrary number belonging to the interval $(1,u)$. In what follows, we will also need the parameter $s=s(p,q,\alpha)$ given by \eqref{defs} and the number 
$$ t=t(r,p,\alpha)=\frac{rpq}{(r-1)(q-p)+pq}.$$
Suppose that $w$ is a weight satisfying the condition $A_r$ and let $\varphi:\R^d\to \R$ be a function belonging to $L^p(w^{p/q})$. Of course, we may assume that $\varphi\geq 0$: in \eqref{mainin}, the passage $\varphi\to |\varphi|$ does not affect the $L^p$-norm of $\varphi$, and does not decrease the $L^q$ norm of $\mathcal{M}^\alpha\varphi$. Furthermore, it is enough to deal with bounded $\varphi$'s only, by a standard truncation argument. Next, let $\mathfrak{Q}$ be a fixed dyadic cube and, for each $n\geq 0$, let $\mathcal{Q}^n$ be the collection of all dyadic cubes of measure $|\mathfrak{Q}|/2^{nd}$, contained within $\mathfrak{Q}$. 
Consider the sequences $(\varphi_n)_{n\geq 0}$, $(\mathtt{w}_n)_{n\geq 0}$ and $(\mathtt{v}_n)_{n\geq 0}$ of conditional expectations of $\varphi$, $w$ and $w^{-1/(r-1)}$ with respect to $(\mathcal{Q}^n)_{n\geq 0}$. That is, in analogy with the previous section, for any $x\in \mathfrak{Q}$ and any nonnegative integer $n$, define
\begin{equation}\label{deffn}
 \varphi_n(x)=\frac{1}{|Q|}\int_Q \varphi,\quad \mathtt{w}_n(x)=\frac{1}{|Q|}\int_Q w\quad \mbox{ and }\quad \mathtt{v}_n(x)=\frac{1}{|Q|}\int_Q w^{-1/(r-1)},
\end{equation}
where $Q$ is the element of $\mathcal{Q}^n$ which contains $x$. 
We will also use the notation
$$ \psi_n(x)=|\mathfrak{Q}|^{\alpha/d}\max\big\{\varphi_0,2^{-\alpha}\varphi_1,2^{-2\alpha}\varphi_2,\,\ldots,\,2^{-n\alpha}\varphi_n\big\}.$$

Before we proceed, let us comment on the reasoning we are going to present. As in the preceding section, the proof will be based on the properties of a certain special function $B$. This time we have four parameters involved, corresponding to the sequences $(\varphi_n)_{n\geq 0}$, $(\psi_n)_{n\geq 0}$, $(\mathtt{w}_n)_{n\geq 0}$ and $(\mathtt{v}_n)_{n\geq 0}$, and thus it is natural to define $B$ on the four-dimensional domain
$$ \mathfrak{D}=\{(x,y,w,v):x\geq 0,\,y\geq 0,\,w>0,\,1\leq wv^{r-1}\leq c\}.$$
Here and below, we use the same letter ``$w$'' to denote the weight and the third coordinate; as we hope, this should not cause any confusion and it should be clear from the context which object we are using at the moment.  
A natural \emph{idea} to follow is to find $B$ for which the sequence $\left(\int_{\mathfrak{Q}} B(\varphi_n,\psi_n,\mathtt{w}_n,\mathtt{v}_n)\mbox{d}z\right)_{n\geq 0}$ is nonincreasing, nonpositive and satisfies the majorization of the form
\begin{equation*}
\begin{split}
 &\liminf_{n\to \infty} \int_{\mathfrak{Q}}  B(\varphi_n,\psi_n,\mathtt{w}_n,\mathtt{v}_n)\mbox{d}z \\
&\qquad \qquad \qquad \geq ||\mathcal{M}^\alpha_d\varphi||_{L^q(\mathfrak{Q};w)}^q-\left(\frac{q}{s}\right)^{q/(q-s)}c^{q/(q-s)}||\varphi||_{L^p(\mathfrak{Q};w^{p/q})}^q
\end{split}
\end{equation*}
or, after some standard limiting arguments (Lebesgue's differentiation theorem, Fatou's lemma, etc.),
\begin{equation}\label{major}
  \int_{\mathfrak{Q}}  B(\varphi,\mathcal{M}^\alpha_d\varphi,w,w^{-1/(r-1)})\mbox{d}z \geq ||\mathcal{M}^\alpha_d\varphi||_{L^q(\mathfrak{Q};w)}^q-\left(\frac{cq}{s}\right)^{q/(q-s)}||\varphi||_{L^p(\mathfrak{Q};w^{p/q})}^q.
\end{equation}
A typical approach in the study of such majorization is to prove the corresponding \emph{pointwise} bound. The problem is that the right hand side above is a mixture of $L^p$ and $L^q$ norms and is \emph{not} of integral form. That is, there seems to be no pointwise inequality which after integration would yield the above majorization. In addition, no manipulations with the inequality \eqref{major} (for instance, replacing the right hand side by
$$ ||\mathcal{M}^\alpha_d\varphi||_{L^q(\mathfrak{Q};w)}-\left(\frac{q}{s}\right)^{1/(q-s)}c^{1/(q-s)}||\varphi||_{L^p(\mathfrak{Q};w^{p/q})}$$
or other expressions of this type) 
seem to lead to the convenient majorization in the integral form. 
To overcome this difficulty, we will make use of the following novel argument which, to the best of our knowledge, has not been used before. Namely, to establish the inequality \eqref{mainin}, we will work with a special function $B$ which itself depends on $\varphi$. Specifically, put 
$$ B(x,y,w,v)=B_{\varphi}(x,y,w,v)=y^qw-\frac{q}{s}||\varphi||_{L^p(\mathfrak{Q};w^{p/q})}^{q-s-t}cx^ty^sv^{1-t}.$$

In the two lemmas below, we study certain crucial properties of this object. We start with the following monotonicity condition.

\begin{lemma}\label{mainlemma}
Suppose that $[w]_{A_r}= c$. Then for any nonnegative integer $n$ we have
\begin{equation}\label{conv}
\int_\mathfrak{Q}B(\varphi_n,\psi_n,\mathtt{w}_n,\mathtt{v}_n)\mbox{d}x\leq \int_\mathfrak{Q} B(\varphi_0,\psi_0,\mathtt{w}_0,\mathtt{v}_0)\mbox{d}x.
\end{equation}
\end{lemma}
\begin{proof} Note that $(\varphi_n,\psi_n,\mathtt{w}_n,\mathtt{v}_n)\in \mathfrak{D}$ (so the composition $B(\varphi_n,\psi_n,\mathtt{w}_n,\mathtt{v}_n)$ makes sense) due to assumption $[w]_{A_r}= c$. Clearly, it suffices to show the inequality
$$ \int_\mathfrak{Q}B(\varphi_{n},\psi_{n},\mathtt{w}_{n},\mathtt{v}_{n})\mbox{d}z\leq \int_\mathfrak{Q} B(\varphi_{n-1},\psi_{n-1},\mathtt{w}_{n-1},\mathtt{v}_{n-1})\mbox{d}z.$$
This will be done in the several separate steps below.

\smallskip

\emph{Step 1.} First we will show the pointwise estimate
\begin{equation}\label{convaux}
\begin{split}
&B\big(\varphi_{n}(z),\psi_{n}(z),\mathtt{w}_{n}(z),\mathtt{v}_{n}(z)\big)\\
&\qquad \qquad \leq B\big(\varphi_{n}(z),\psi_{n-1}(z),\mathtt{w}_{n}(z),\mathtt{v}_{n}(z)\big),\qquad z\in \mathfrak{Q}.
\end{split}
\end{equation}
Let $Q^n(z)$ be the element of $\mathcal{Q}^{n}$ which contains $z$. 
The above bound is trivial when $\psi_{n}(z)>|Q^n(z)|^{\alpha/d}\varphi_{n}(z)=|\mathfrak{Q}|^{\alpha/d}2^{-n\alpha}\varphi_{n}(z)$, since then $\psi_{n}(z)=\psi_{n-1}(z)$. Suppose that $\psi_{n}(z)\leq |Q^n(z)|^{\alpha/d}\varphi_{n}(z)$ (so actually we have equality: see the definition of the sequence $\psi$). Since $\psi_{n}(z)\geq \psi_{n-1}(z)$, we will be done if we show that 
\begin{equation}\label{part_ineq}
 \frac{\partial }{\partial y}B\big(\varphi_{n}(z),y,\mathtt{w}_{n}(z),\mathtt{v}_{n}(z)\big)\leq 0
\end{equation}
for $y\in \big(0,|Q^n(z)|^{\alpha/d}\varphi_{n}(z)\big)$, $n=0,\,1,\,2,\,\ldots$. The partial derivative equals
$$ qy^{s-1}\mathtt{w}_{n}(z)\left[y^{q-s}-||\varphi||_{L^p(\mathfrak{Q};w^{p/q})}^{q-s-t}c\varphi_{n}(z)^t\mathtt{w}_{n}(z)^{-1}\mathtt{v}_{n}(z)^{1-t}\right].$$
However, we have $\mathtt{w}_{n}(z)\mathtt{v}_{n}(z)^{r-1}\leq c$, directly from the assumption $[w]_{A_r}=c$. Furthermore, we have $y<|Q^n(z)|^{\alpha/d}\varphi_{n}(z)$, which has been just imposed above. Consequently, we see that the partial derivative in \eqref{part_ineq} is not larger than
$$ qy^{s-1}\mathtt{w}_{n}(z)\varphi_{n}(z)\left[|Q^n(z)|^{\alpha(q-s)/d}\varphi_{n}(z)^{q-s-t}-||\varphi||_{L^p(\mathfrak{Q};w^{p/q})}^{q-s-t}\mathtt{v}_{n}(z)^{r-t}\right].$$
However, by H\"older's inequality, we may write
\begin{align*}
& |Q^n(z)|\varphi_{n}(z)
=\int_{Q^n(z)} \varphi\\
&\leq \left(\int_{Q^n(z)} \varphi^p w^{p/q}\right)^{1/p}\left(\int_{Q^n(z)}w^{-1/(r-1)}\right)^{(r-1)/q}|Q^n(z)|^{1-1/p-(r-1)/q}\\
&\leq ||\varphi||_{L^p(\mathfrak{Q};w^{p/q})}\mathtt{v}_{n}(z)^{(r-1)/q}|Q^n(z)|^{1-1/p-(r-1)/q},
\end{align*}
which is equivalent to
$$ |Q^n(z)|^{\alpha(q-s)/d}\varphi_{n}(z)^{q-s-t}-||\varphi||_{L^p(\mathfrak{Q};w^{p/q})}^{q-s-t}\mathtt{v}_{n}(z)^{r-t}\leq 0.$$
This shows that \eqref{part_ineq}, and hence also \eqref{convaux}, hold true. 

\smallskip

\emph{Step 2.} Now, observe that the $C^1$ function $G:[0,\infty)\times (0,\infty)\to [0,\infty)$ given by $G(x,v)= x^tv^{1-t}$ is convex. This is straightforward: for $x,\,v>0$, the Hessian matrix
$$ D^2G(x,v)=\left[\begin{array}{cc}
t(t-1)x^{t-2}v^{1-t} & t(1-t)x^{t-1}v^{-t}\\
r(1-t)x^{t-1}v^{-t} & t(t-1)x^t v^{-1-t}
\end{array}\right]$$
is nonnegative-definite.

\smallskip

\emph{Step 3.} We are ready to establish the desired bound \eqref{conv}. Pick an arbitrary element $Q$ of $\mathcal{Q}^{n-1}$ and apply \eqref{convaux} to get 
  $$ \int_QB(\varphi_{n},\psi_{n},\mathtt{w}_{n},\mathtt{v}_{n})\mbox{d}z\leq 
 \int_QB(\varphi_{n},\psi_{n-1},\mathtt{w}_{n},\mathtt{v}_{n})\mbox{d}z.$$
 Directly from the definition of the sequences $(\varphi_n)_{n\geq 0}$, $(\mathtt{w}_n)_{n\geq 0}$ and $(\mathtt{v}_{n})_{n\geq 0}$, we see that $\varphi_{n-1}$, $\mathtt{w}_{n-1}$ and $\mathtt{v}_{n-1}$ are constant on $Q$ and equal to
 $$ \frac{1}{|Q|}\int_Q \varphi_{n}\mbox{d}z,\qquad \frac{1}{|Q|}\int_Q \mathtt{w}_{n}\mbox{d}z\qquad \mbox{and}\qquad \frac{1}{|Q|}\int_Q \mathtt{v}_{n}\mbox{d}z$$
there, respectively. Now, we use Step 2 and the formula for $B$ to get
$$ \int_QB(\varphi_{n},\psi_{n-1},\mathtt{w}_{n},\mathtt{v}_{n})\mbox{d}z\leq 
\int_QB(\varphi_{n-1},\psi_{n-1},\mathtt{w}_{n-1},\mathtt{v}_{n-1})\mbox{d}z$$
and thus
$$ \int_QB(\varphi_{n},\psi_{n},\mathtt{w}_{n},\mathtt{v}_{n})\mbox{d}z\leq \int_QB(\varphi_{n-1},\psi_{n-1},\mathtt{w}_{n-1},\mathtt{v}_{n-1})\mbox{d}z.$$
It remains to sum over all $Q\in \mathcal{Q}^{n-1}$ to get the claim.
\end{proof}

We will also require the following properties.

\begin{lemma}\label{easyprop}
(i) If $(x,y,w,v)\in \mathfrak{D}$ satisfies $y=|\mathfrak{Q}|^{\alpha/d}x$, then we have the pointwise inequality
$$
B(x,y,w,v)\leq 0.
$$

(ii) For any $(x,y,w,v)\in \mathfrak{D}$ we have the majorization
$$
B(x,y,w,v)\geq \frac{q-s}{q}\left[ y^pw-\left(\frac{q}{s}\right)^{q/(q-s)}||\varphi||_{L^p(\mathfrak{Q};w^{p/q})}^{q(q-s-t)/(q-s)}c^{q/(q-s)}(xw^{1/q})^p\right].
$$
\end{lemma}
\begin{proof}
(i) This follows from \eqref{part_ineq} with $n=0$. Indeed, for any $w,\,v$ satisfying $1\leq wv\leq c$ there is a weight $w$ with $[w]_{A_r}\leq c$ satisfying $\mathtt{w}_0=w$, $\mathtt{v}_0=v$ (see e.g. \cite{V}). Thus, if we put $\varphi\equiv x$, then $\varphi_0=x$, $\psi_0=y$ and hence \eqref{part_ineq} gives
$$ B(x,y,w,v)\leq B(\varphi_0(z),0,\mathtt{w}_0(z),\mathtt{v}_0(z))=0.$$

(ii) Of course, it suffices to show the bound for $y>0$. By the mean value property, for any $\beta\geq 0$ we have
$$ \beta^{q/(q-s)}-1\geq \frac{q}{q-s}(\beta-1).$$
Plugging 
$$ \beta=\frac{q}{s}||\varphi||_{L^p(w^{p/q})}^{q-s-t}cx^ty^{s-q}w^{(p-q)(q-s)/q^2+1}$$
and multiplying both sides by $y^qw$, we obtain an inequality which is equivalent to 
\begin{align*}
& B(x,y,w,w^{1/(1-r)})\geq \\
&\frac{q-s}{q}\left[ y^pw-\left(\frac{q}{s}\right)^{q/(q-s)}||\varphi||_{L^p(\mathfrak{Q};w^{p/q})}^{q(q-s-t)/(q-s)}c^{q/(q-s)}(xw^{1/q})^p\right].
\end{align*}
It remains to observe that $B(x,y,w,v)$ increases when $v$ increases, and therefore we have $B(x,y,w,v)\geq B(x,y,w,w^{1/(1-r)})$.
\end{proof}

We are ready to establish our main result.

\begin{proof}[Proof of \eqref{mainin}] Combining the previous two lemmas, we obtain the bound
 \begin{align*}
& \int_{\mathfrak{Q}} \psi_n(z)^q \mathtt{w}_n(z)\mbox{d}z\\
&\leq  \left(\frac{q}{s}\right)^{q/(q-s)}||\varphi||_{L^p(\mathfrak{Q};w^{p/q})}^{q(q-s-t)/(q-s)}c^{q/(q-s)}\int_\mathfrak{Q}\varphi_n(z)^p\mathtt{w}_n(z)^{p/q}\mbox{d}z.
\end{align*}
All that is left is to carry out appropriate limiting procedure. First, let $n$ go to infinity. Then $\psi_n$ increases to
$$ \psi_\infty=\sup\!\left\{\!|Q|^{\alpha-1}\int_Q \varphi\;:\;x\in Q, Q \in\mathcal{Q}^k\mbox{ for some }k\right\}.$$
In addition, we have $\varphi_n\to \varphi$ almost everywhere (by Lebesgue's differentiation theorem) and $\mathtt{w}_n\to w$ in $L^1(\mathfrak{Q})$, by the standard facts concerning conditional expectations (see e.g. Doob \cite{Do}). Putting these facts together, and combining them with the boundedness of $\varphi$ assumed at the beginning, we get
\begin{align*}
 \int_\mathfrak{Q} \psi_\infty^q w&\leq \liminf_{n\to \infty} \int_\mathfrak{Q} \psi_n^q \mathtt{w}_n\\
 &\leq \left(\frac{q}{s}\right)^{q/(q-s)}||\varphi||_{L^p(w^{p/q})}^{q(q-s-t)/(q-s)}c^{q/(q-s)}\liminf_{n\to\infty}\int_\mathfrak{Q}\varphi_n^p\mathtt{w}_n^{p/q}\\
 &\leq  \left(\frac{q}{s}\right)^{q/(q-s)}||\varphi||_{L^p(w^{p/q})}^{q(q-s-t)/(q-s)}c^{q/(q-s)}||\varphi||_{L^p(w^{p/q})}\\
 &=\left(\frac{q}{s}\right)^{q/(q-s)}c^{q/(q-s)}||\varphi||_{L^p(w^{p/q})}^q.
\end{align*}
Next, assume that $\mathfrak{Q}_1\subseteq \mathfrak{Q}_2\subseteq \ldots$ is a strictly increasing sequence of dyadic cubes, and apply the above estimate with respect to $\mathfrak{Q}=\mathfrak{Q}_n$. Then, as $n\to \infty$, we have $\psi_\infty\uparrow \mathcal{M}^\alpha_d\varphi$ and hence \eqref{mainin} follows from Lebesgue's monotone convergence theorem.
\end{proof}

\section{A weighted version of Doob's inequality}

In the particular case $\alpha=0$, Theorem \ref{mainth} gives the following statement for Hardy-Littlewood maximal operator.
\begin{theorem}\label{mainth0}
Let $w$ be an $A_p$ weight, $1<p<\infty$. Then for any locally integrable function $\varphi$ on $\R^d$ we have
\begin{equation}\label{mainin11}
 ||\mathcal{M}_d\varphi||_{L^p(w)}\leq \inf_{1< r<p}\left\{[w]_{A_r}^{1/r}\left(\frac{p}{p-r}\right)^{1/r}\right\}||\varphi||_{L^p(w)}.
\end{equation}
\end{theorem}

This theorem can be given a probabilistic meaning, in terms of martingales, and such a result was proved in \cite{IK}. Let us recall the general setting of $A_p$ martingales. Suppose that $(\Omega, \calF, \Pro)$ is a probability space with a non-decreasing right continuous
family $(\calF_t)_{t\geq 0}$ of sub $\sigma$-fields of $\calF$ such that $\calF_0$  contains all $\Pro$-null sets. 
Fix a random variable $Z$ such that $Z >0$ a.s. and such that $\E[Z]=1$. With $p>1$, we say that the random variable $Z$ is an $A_p$-weight  if 
\begin{equation}\label{probAp}
[Z]_{A_p}:=\sup_{t}\Big\|Z_{t}\left(\E\big[\left(\frac{1}{Z}\right)^{1/(p-1)}\big|\calF_t\big]\right)^{p-1}\Big\|_{L^{\infty}}<\infty,
\end{equation}
where $Z_t=\E\left[Z\, |\calF_t\right]$.  The following result can be derived by keeping track of the constants in the proof of  \cite[Theorem 2]{IK}.

\begin{theorem}\label{MuIzuKaz} Let $X_t=\E\left[X\, |\calF_t\right]$ be the martingale generated by the $\Pro$--integrable random variable $X$.  Suppose $Z\in A_{r}$, for some $r>1$.  Then for all $p>r$, 
\begin{equation}\label{weighteddoob}
\|X^{*}\|_{L^p(\hat{\Pro})}\leq [Z]_{A_r}^{1/r} \left(\frac{p}{p-r}\right)^{1/r} \|X\|_{L^p(\hat{\Pro})}, 
\end{equation}
where as usual $X^*=\sup_{t}|X_t|$ and $d\hat{\Pro}$ denotes the the measure $Zd\Pro$. 
\end{theorem}

We give the proof of this result since it is simple and short.  We assume, as we may, that $\|X\|_{L^p(\hat{\Pro})}<\infty$.  A simple calculation (just the definition of conditional expectation) gives that 
$$
\hat{\E}[X\,|\calF_t]=\frac{\E[ZX\,|\calF_t]}{Z_t}
$$
and this holds almost surely with respect to both probability measures $\Pro$ and $\hat{\Pro}$.  Applying this with 
$(1/Z)X$ in place of $X$ we see that 
$$
X_t=Z_t\hat{\E}[(1/Z)X\,|\calF_t]. 
$$
If we let $r_0$ be the conjugate exponent of $r$, H\"older's inequality gives 
\begin{eqnarray*}
|X_t|^{r} &\leq &\left\{Z_t^{r}\left[\hat{\E}\left[\left(\frac{1}{Z} \right)^{r_0}\,|\calF_t \right] \right]^{r -1} 
 \right\}\hat{\E}[|X|^{r}\, |\calF_t]\\
 &=& Z_t\left[\E\left[\left(\frac{1}{Z} \right)^{r_0-1}\,|\calF_t \right] \right]^{r -1}\hat{\E}[|X|^{r}\, |\calF_t]\\
 &\leq &  [Z]_{A_r}\hat{\E}[|X|^{r}\, |\calF_t]
\end{eqnarray*}
 Now, applying Doob's inequity with $p/{r}>1$ to the $\hat{\Pro}$ martingale in the second term gives the inequality.

As Theorem \ref{mainth}, this proposition can also be obtained using the Bellman functions techniques as above, bypassing Doob's inequality. 

We now address the martingale version of Theorem \ref{weights}. While Theorem \ref{MuIzuKaz} holds for arbitrary martingales,  it is proved in Uchiyama \cite{Uch} that the $A_{p-\varepsilon}$ result does not hold in general for arbitrary martingales but it does so if the martingales have continuous trajectories (Brownian martingales) or if they are regular. Recall that the nondecreasing  filtration $(\calF_n)_{n\geq 0}$ on the probability space $\left(\Omega, \calF, \Pro\right)$ with ${\bigvee}_{n=1}^{\infty}\calF_n=\calF$ is said to be regular if $\calF_n$ is atomic for each $n$ and there is a universal constant $C_0$ such that $\Pro(A)/\Pro(B)< C_0$ for any two atoms $A\in \calF_{n-1}$ and $B\in \calF_n$ such that $B\subset A$. Dyatic filtrations on $\R^d$ are for example regular with $C_0=2^d$.  The proof of Theorem \ref{weights} applies to $A_p$ weights on a regular filtration and we obtain the following theorem which gives a  martingale version of  Buckley's inequality. 

\begin{theorem} Let $X_n=\E\left[X\, |\calF_n\right]$ be the martingale generated by the $\Pro$--integrable random variable $X$ and assume that $(\calF_n)_{n\geq 0}$ is regular with constant  $C_0$.  Suppose $Z\in A_{p}$, $1<p<\infty$. There is a constant $C$ depending on $C_0$ such that 
\begin{equation}\label{weightedoob1}
\|X^{*}\|_{L^p(\hat{\Pro})}\leq \frac{Cp}{p-1}[Z]_{A_p}^{1/(p-1)} \|X\|_{L^p(\hat{\Pro})}.   
\end{equation}
Except for the constant $C$, this inequality is sharp.   The same result holds for martingales with continuous trajectories for some universal constant $C$. 
\end{theorem}

\section*{Acknowledgment} The results of this paper were obtained during the fall semester of 2013 when the second-named author visited Purdue University.

\end{document}